\documentclass[a4paper,12pt]{amsart}
\usepackage{amssymb}
\usepackage{amsmath}
\usepackage{stmaryrd}
\usepackage{amscd,amsthm,amssymb}
\usepackage{enumerate}
\usepackage{color}

\scrollmode
\usepackage{latexsym}

\def\.{\cdot}

\def\d{{\delta}}

\def\vs{\vskip .6cm}

\def\la{\langle}
\def\ra{\rangle}

\def\beq{\begin{equation}}
\def\eeq{\end{equation}}
\def\bea{\begin{eqnarray*}}
\def\eea{\end{eqnarray*}}
\def\beaa{\begin{eqnarray}}
\def\eeaa{\end{eqnarray}}
\def\ba{\begin{array}}
\def\ea{\end{array}}

\def\L{\Lambda}

\def\be{\begin{equation}}
\def\ee{\end{equation}}

\def\Sym{\mathrm{Sym}}

\def\R{\mathrm{R}}

\def\Sym{\mathrm{Sym}}

\def\T{\mathrm{\,T}}

\def\lr{\,\lrcorner\,}

\def\dd{\mathrm{d}}
\def\LL{\mathrm{L}}

\def\LieD{ \mathcal{L}}
\def \R{\mathbb{R}}

\newcommand{\floor}[1]{\left\lfloor#1\right\rfloor}
\newcommand{\set}[1]{\left\{#1\right\}} 
\newcommand{\parens}[1]{\left( #1 \right)}

\newcommand{\suchthat}{\;\middle|\;}

\newtheorem{epr}{Proposition}[section]
\newtheorem{ath}[epr]{Theorem}
\newtheorem{elem}[epr]{Lemma}

\theoremstyle{definition}

\title[Killing tensors on tori]{Killing tensors on tori}

\address{Konstantin Heil\\
Institut f\"ur Geometrie und Topologie \\
Fachbereich Mathematik\\
Universit{\"a}t Stuttgart\\
Pfaffenwaldring 57 \\
70569 Stuttgart, Germany
}
\email{konstantin.heil@mathematik.uni-stuttgart.de}

\author{Konstantin Heil, Andrei Moroianu, Uwe Semmelmann}

\address{Andrei Moroianu \\ Laboratoire de Math\'ematiques de Versailles, UVSQ, CNRS, 
Universit\'e Paris-Saclay, 78035 Versailles, France }
\email{andrei.moroianu@math.cnrs.fr}

\address{Uwe Semmelmann\\
Institut f\"ur Geometrie und Topologie \\
Fachbereich Mathematik\\
Universit{\"a}t Stuttgart\\
Pfaffenwaldring 57 \\
70569 Stuttgart, Germany
}
\email{uwe.semmelmann@mathematik.uni-stuttgart.de}

\date{\today}

\begin{document}

\begin{abstract}
We show that Killing tensors on conformally flat $n$-dimensional tori whose conformal factor only depends on one variable, are polynomials in the metric and in the Killing vector fields. In other words, every first integral of the geodesic flow polynomial in the momenta on the sphere bundle of such a torus is linear in the momenta.
\vs
\smallskip

\end{abstract}

\maketitle

\section{Introduction}

Killing tensors are symmetric $p$-tensors with vanishing symmetrized covariant derivative and correspond to Killing vector fields for $p=1$. 
Originally, Killing tensors were studied in the physics literature since they define first integrals (polynomial in the momenta)
of the equation of motion, and thus functions constant on geodesics. This property makes Killing tensors very 
important in the theory of integrable systems.

First integrals of the geodesic flow on the $2$-dimensional torus is an intensively studied topic. 
The description of metrics with linear first integrals, i.e. Killing vector fields, is obvious. There also is  
a classification of metrics with quadratic first integrals, i.e. with Killing (non-parallel) $2$-tensors. These metrics 
turn out to be of Liouville type, i.e. in conformal coordinates the metric can be written as 
$\tilde g=(f(x)^2 + g(y)^2)(\dd x^2 + \dd y^2)$  cf. \cite{kolo82}. Surprisingly, the existence of
first integrals of degree $\ge 3$ independent of those of degree $1$ and $2$ on a 2-torus is a completely open problem. The conjecture is 
that there are no such first integrals cf.  \cite{bofoma98}. During the  last thirty years many partial results were proved 
supporting this conjecture cf.  \cite{bia87}, \cite{koz85}, \cite{koz95}.

It follows from the theory of integrable systems that besides the metric, it is not possible to have two other functional independent first integrals on the 2-torus cf. \cite{mat11}. Indeed, in this situation the geodesic flow would be {\it superintegrable}. Then its trajectories would lie in the intersection 
of three level sets, one of which is the compact sphere bundle. One can conclude that all geodesics have to be closed. Then the manifold has the homology ring of a rank one symmetric space, which 
is not the case for the torus. 

In particular this means that if the metric $\tilde g$ on $T^2$ carries a Killing vector field, then any Killing tensor of
higher degree is functional dependent of it, i.e. it is expressible as polynomial in the Killing 
vector field and the metric. Since every metric $\tilde g$ on $T^2$ is conformal to the flat metric $g$, and has a Killing vector field if and only if the conformal factor only depends on one coordinate of $T^2$, the above fact can be equivalently stated as follows: {\em Every Killing tensor on the torus $T^2$ equipped with a metric of the form $\tilde g= e^{2f(y)}(\dd x^2 +  \dd y^2)$ is a polynomial in the Killing vector field $\frac\partial{\partial x}$ and in the metric tensor $\tilde g$. }

In our article we generalize this fact to the $n$-dimensional torus. Our main result is

\begin{ath} \label{main}
Let $K$ be a Killing tensor on the torus $T^n$ equipped with a metric of the form $\tilde g= e^{2f(x_n)}(\dd x_1^2 + \ldots + \dd x_n^2)$.
Then $K$ is a polynomial in the Killing vector fields $\frac\partial{\partial x_1},\ldots, \frac\partial{\partial x_{n-1}}$ and in the metric tensor $\tilde g$. 
\end{ath}

The idea of the proof is as follows. We first translate the Killing equation on the flat torus. Next, using the formalism developed in \cite{aku}, we show that the components of any Killing tensor with respect to the flat metric are constant functions in $x_1,\ldots,x_{n-1}$. The Killing equation then reduces to a system of ordinary differential equations with polynomial solutions, which translated back to the metric $\tilde g$ yields the result. 

Note that for dimensional reasons, the result above cannot be proved by the above arguments from the theory of integrable systems for $n\ge 3$. Indeed, assuming that $K$ is a Killing tensor on $(T^n,\tilde g)$, functionally independent on the Killing vector fields $\xi_1,\ldots,\xi_{n-1}$ and on the metric tensor $\tilde g$, then one would obtain $n+1$ first integrals of the geodesic flow on the tangent bundle of $(T^n,\tilde g)$, but this no longer implies superintegrability since $n+1<2n-1$ for $n\ge 3$.

{\sc Acknowledgments.} This work was initiated during a ``Research in Pairs'' stay at the Mathematisches Forschungsinstitut, Oberwolfach, Germany and  partially supported by the Procope Project No. 32977YJ. We also thank Vladimir Matveev for very useful comments.

\section{Preliminaries}

We will use the formalism introduced in our article \cite{aku}. For the convenience of the
reader, we recall here the standard definitions and formulas which are relevant in the sequel.

Let $(\T M, g)$ be the tangent bundle of a $n$-dimensional Riemannian manifold $(M,g)$. We denote with
$\Sym^p \T M \subset \T M^{\otimes p}$ the $p$-fold symmetric tensor product of $\T M$. The elements of $\Sym^p \T M$
are linear combinations of symmetrized tensor products
$$
v_1 \cdot \ldots \cdot   v_p := \sum_{\sigma \in S_p} \, v_{\sigma(1)} \otimes \ldots \otimes v_{\sigma(p)} \ ,
$$
where $v_1, \ldots, v_p$ are tangent vectors in $\T M$. 

Let $\{e_i\}$ denote from now on a local orthonormal frame of $(\T M, g)$. Using the metric $g$, we will identify $\T M$ with $\T^* M$ and thus $\Sym^2 \T^* M\simeq\Sym^2\T M$. Under this identification we view the metric tensor as a symmetric 2-tensor
$\LL := 2 g = \sum e_i \cdot e_i$.
The scalar product $g$ induces a scalar product, also denoted by $g$, on $\Sym^p \T M$ defined by 
$$
g(v_1 \cdot \ldots \cdot v_p,w_1 \cdot \ldots \cdot w_p) \;:=\; \sum_{\sigma \in S_p}
\, g(v_1, w_{\sigma(1)}) \cdot \ldots  \cdot g(v_p, w_{\sigma(p)})\ .
$$
Using this scalar product, every element $K$ of $\Sym^p \T M$ can be identified with  a polynomial map of degree $p$ on $\T M$,
defined  by the formula 
$K(v_1,\ldots,v_p)=g(K,v_1\cdot\ldots\cdot v_p) .
$
The metric adjoint of the bundle homomorphism
$
v\cdot: \Sym^p \T M \rightarrow \Sym^{p+1} \T M, \;  K \mapsto v \cdot K
$
is the contraction map
$
v\lrcorner : \Sym^{p+1} \T M \rightarrow \Sym^{p} \T M, \;  K \mapsto v \lrcorner \, K
$, defined by $$(v \lrcorner \, K) (v_1, \ldots , v_{p-1}): = K(v, v_1, \ldots, v_{p-1})\ .$$

\noindent
The metric adjoint of $L\cdot:\Sym^{p} \T M \rightarrow \Sym^{p+2} \T M$ is the bundle homomorphism
$$
\L : \Sym^{p+2}\T M \rightarrow \Sym^{p} \T M , \quad K \mapsto  \sum e_i \lrcorner \,  e_i \lrcorner \,  K\ .
$$
The following commutator formulas are straightforward:
\beq \label{commu2}
[\, \Lambda, \, v \,\cdot \, ] \;=\; 2\, v \,  \lrcorner  \, , \quad
[\, v \lrcorner \,,\, \LL  \,  \,] \;=\; 2\, v \cdot\, ,\quad
[\,\L , \, v \lrcorner \, \,] \;=\; 0 \;=\; [\, \LL , \, v \cdot \,] \ . 
\eeq
For later use, let us state the following formula which holds for any vector $v\in \T M$ and symmetric tensor  $K\in \Sym^q (\T M)$:
\beq\label{power}
(L \cdot K)(v, \ldots, v) = (q+2)(q+1) K(v, \ldots, v) |v|^2 \ .
\eeq
Indeed, using \eqref{commu2} repeatedly we may write
$$
(\LL \cdot K) (v, \ldots, v) = \la \LL \cdot K, v^{q+2}\ra = \la K, \L v^{q+2}\ra = (q+2)(q+1) \la K, v^q\ra  |v|^2\ .
$$

We denote by
$
\Sym^p_0 \T M := \ker( \L : \Sym^p \T M \rightarrow \Sym^{p-2} \T M)
$
the space of trace-free symmetric $p$-tensors. The bundle of symmetric tensors splits as 
$$
\Sym^p \T M  \;\cong\;  \Sym^p_0 \T M \;\oplus\; \Sym^{p-2}_0 \T M \;\oplus\;  \ldots \ ,
$$
where the last summand in the decomposition is the trivial rank one bundle for $p$ even and the tangent bundle $\T M$ for $p$ odd.
Correspondingly we have for  any $K \in \Sym^p \T M$ the
decomposition 
$$
K \;=\;  K_0 \;+\;  \LL K_1 \;+\;  \LL ^2K_2 \;+ \; \ldots
$$
with $K_i \in \Sym^{p-2i}_0 \T M$, i.e. $\Lambda K_i = 0$, which is called the {\it standard decomposition} of $K$.
For any $v \in \T M$ and $K \in \Sym^p_0 \T M$ the following projection formula holds (cf. Eq. (3) in \cite{aku}):
\beq \label{projection}
(v \cdot K )_0 \;=\; v \cdot K \;-\; \tfrac{1}{n + 2 (p-1)}\, \LL  \,  (v \lrcorner \, K)\ .
\eeq
On sections of $\Sym^p \T M$ we define two first order differential operators, the differential
$$
\dd : \Gamma(\Sym^p \T M) \rightarrow \Gamma( \Sym^{p+1}\T M), \quad K \mapsto \sum e_i \cdot \nabla_{e_i}K \ ,
$$
and the formal adjoint of 
$\dd$,  the divergence operator  $\d$:
$$
\d : \Gamma(\Sym^{p+1} \T M) \rightarrow \Gamma(\Sym^{p}\T M), \quad K \mapsto -  \sum e_i \lrcorner\, \nabla_{e_i}K \ ,
$$
An important property of $\dd$ is that it acts as derivation on symmetric products. Moreover, $\dd$ commutes with $\LL\cdot$ and
for a section $K$ of $\Sym^p_0 \T M$ the projection of  $\dd K$ onto the trace-free part is given by the following formula (cf. \cite {aku}, Lemma 2.2):
\beq\label{dprojection}
(\dd K)_0 \;=\; \dd K \;+\; \tfrac{1}{n+2(p-1)}\,\LL  \, \d K \ .
\eeq

A symmetric tensor $K \in \Gamma(\Sym^p \T M)$ is called {\em conformal Killing tensor} if
there exists some symmetric tensor $k \in \Gamma( \Sym^{p-1}\T M)$ with
$\;
\dd  K = \LL  \cdot k
$.
It is called  {\em Killing tensor} if $\dd K = 0$ or equivalently if the symmetrized covariant derivative of $K$ 
vanishes or if $(\nabla_X K)(X,\ldots, X)= 0$ holds for all vector fields $X$. A Killing tensor is in particular a
conformal Killing tensor.

The defining equation for conformal Killing tensors is conformally invariant, i.e. a section $K$ of $ \Sym^{p} \T M$ is a conformal Killing tensor
with respect to the metric $g$, if and only if it is a conformal Killing tensor with respect to every conformally related metric $  \tilde g = e^{2 f} g$.
Indeed (cf. \cite{aku}, Lemma 3.3), the differential $\tilde \dd$ with respect to the metric $\tilde g$ is related to $\dd$ by 
\beq\label{dconformal}
\tilde \dd K \;=\; e^{-2f} \,(\dd K \,+\, \LL \cdot \dd f \lr K) \ .
\eeq

\section{Proof of the Theorem \ref{main}}

We consider the $n$-dimensional torus $T^n$ with the flat metric $g:= \dd x_1^2 + \ldots + \dd x_n^2$ and a conformally 
related metric $\tilde g := e^{2f}g$ such that $f = f(x_n)$ only depends on the last variable. Let $K \in \Gamma (\Sym^p TM)$ be a Killing tensor on $M:=(T^n, \tilde g)$. By \eqref{dconformal}, the Killing equation $\tilde\dd K=0$ translates into the equation
\begin{equation}\label{k0}\dd K = - \LL  \cdot (\dd f \lr K)
\end{equation} 
with respect to the conformally equivalent flat metric $g= e^{-2f}\tilde g$. 

We denote by $\xi_j:=\frac\partial{\partial x_j}$ the vector fields dual to $\dd x_j$, which form a global orthonormal frame on the flat torus $(T^n,g)$. In particular we have $\LL = \sum^n_{j=1}\xi^2_j$. For $j=1,\ldots, n-1$, $\xi_j$ are Killing vector fields with respect to  $\tilde g$ and parallel with respect to the flat metric $g$.

In \eqref{k0} we use the standard decomposition
$
K = \sum_{j \ge 0} \LL^j K_j
$
with $K_j \in \Gamma(\Sym^{p-2j}_0 \T M)$, together with the fact that $\dd$ commutes with $\LL$ and that
$\dd f \lr \LL = 2\, \dd f$, to obtain 
\bea
\sum_{j \ge 0} \, \LL^j  \,\dd K_j
&=&
- \, \LL \, \dd f \lr \sum_{j \ge 0} \, \LL^j \cdot  K_j
\;=\;
- \LL \sum_{j \ge 0} \left( 2 j  \,\dd f\cdot \LL^{j-1} \cdot K_j \;+\; \LL^j \cdot \dd f \lr K_j \right) \\
&=&
- \sum_{j \ge 0} \left(  2j \, \LL^j \cdot \dd f \cdot K_j \;+\; \LL^{j+1} \cdot \dd f \lr K_j \right) \ .
\eea
Using \eqref{dprojection} and \eqref{dconformal} in the equation above and comparing the trace-free coefficients of $\LL^j$ for every $j$ yields the system
\beaa\label{system}
&&
\dd K_j \;+\;\tfrac{1}{n+2(p-2j-1)} \, \LL \cdot \d K_j \;+\; 2j \, (\dd f \cdot K_j \,-\, \tfrac{1}{n+2(p-2j-1)} \, \LL \cdot \dd f \lr K_j) 
\\
&&
\qquad\qquad\qquad\qquad
= \; \tfrac{1}{n+2(p-2j+1)} \, \d K_{j-1} \;-\; \left(1+\tfrac{2(j-1)}{n+2(p-2j+1)}\right) \, \dd f \lr K_{j-1} \ . \nonumber
\eeaa

Here $j$ goes from $0$ to $ \floor{\frac{p}{2}}$, where we set as usual $K_{-1} = K_{\floor{\frac{p}{2}}+1} = 0$. In particular,
for $j=0$, we obtain as first equation
$$
\dd K_0 \;+\; \tfrac{1}{n+2(p-1)} \, \LL \cdot \d K_0 \;=\; 0 \ .
$$
Thus $K_0$ is a trace-free conformal Killing tensor with respect to $g$ and hence parallel because of  Proposition 6.6 from  \cite{aku}.

\begin{elem}\label{properties}
Let $\xi$ be a linear combination with constant coefficients of the $g$-parallel vector fields $\xi_1, \ldots, \xi_{n-1}$. We denote by $'$ the derivative in direction of $\xi_n$.  Then the
following holds
\begin{enumerate}
\item[(i)]
$\nabla_\xi K \;=\; 0$ and $ \nabla_\xi K_j \;=\; 0$ for all
$j$ 
\item[(ii)]
$\dd f = f' \, \xi_n, \quad   \nabla_\xi \,\dd  f \;=\; 0 $   \nonumber    
\item[(iii)]
$\dd K= \xi_n \cdot \nabla_{\xi_n} K, \quad \d K = - \xi_n \lr \nabla_{\xi_n} K$  \nonumber  
\item[(iv)]
$K(\xi, \ldots, \xi, \xi_n)_x \;=\; 0$ at all points $x\in T^n$ where $f'(x) \neq 0$ \nonumber.
\end{enumerate}
\end{elem}
\proof
$(i)$ Since $\nabla K_0 = 0$, $\nabla \xi = 0$ and 
\begin{equation}\label{ld}\LieD_\xi = \nabla_\xi - (\nabla \xi)_* = \nabla_\xi, 
\end{equation}
it follows $\LieD_\xi K_0 =0$.
The Lie derivative $\LieD_\xi$ preserves the space of trace-free tensors as well as that of Killing tensors and of course $\LieD_\xi\LL=0$.
Hence, taking the Lie derivative in the standard decomposition $K = \sum_{j\ge 0}\LL^j K_j$, we obtain that the standard decomposition of the Killing tensor $\LieD_\xi K$ is
\begin{equation}\label{sd}
\LieD_\xi K \;=\; \LL \cdot ( \LieD_\xi K_1 + \LL \cdot  \LieD_\xi K_2 + \ldots  )\ .
\end{equation}
The operator $\LL\cdot$ is injective and commutes with differential $\dd$.
Thus a symmetric tensor $Q$ is Killing if and only if $\LL \cdot Q$ is Killing. From \eqref{sd}, we conclude that
$\LieD_\xi K_1 + \LL \cdot  \LieD_\xi K_2 + \ldots$ is Killing with trace-free part $\LieD_\xi K_1$. Repeating the argument 
above we get $\LieD^2_\xi K_1 = 0$. Taking the scalar product with $K_1$ and integrating over $(T^n,g)$ yields:
$$0=\int_{T^n} g(\LieD^2_\xi K_1,K_1)\dd\mu=\int_{T^n} \LieD_\xi (g(\LieD_\xi K_1,K_1))-g(\LieD_\xi K_1,\LieD_\xi K_1)\dd\mu=-\int_{T^n}|\LieD_\xi K_1|^2\dd\mu$$
since, $\xi$ being parallel, the integral over $M$ of $\xi(\psi)$ vanishes for every function $\psi$ by the Stokes formula. This shows that $\LieD_\xi K_1 = 0$.
By immediate induction using \eqref{sd} we obtain $\LieD_\xi K_j = 0$ for all $j$ and finally also $\LieD_\xi K = 0$. We conclude using \eqref{ld} again.

$(ii)$ The function $f$ only depends on the last coordinate $x_n$, i.e.  $\dd f =f'(x_n) \xi_n$. Thus $\nabla_\xi \,\dd  f \;=\; f'(x_n) \nabla_\xi \,\xi_n \;=\;  0 $. 

$(iii)$ is a direct consequence of $(i)$ and $(ii)$.

$(iv)$
Using $(iii)$ and \eqref{k0}
we get
$
(\xi_n \cdot \nabla_{\xi_n} K)(\xi, \ldots, \xi) = - (\LL \cdot f'\,  \xi_n \lr K )(\xi, \ldots, \xi) 
$.
Since $\xi_n$ is orthogonal to $\xi$, we obtain by \eqref{power} that
$
0 = - p(p+1)f'K(\xi, \ldots, \xi, \xi_n)
$
and the statement follows.
\qed

\bigskip

We now introduce  the functions 
$$\alpha_j := K_j(\xi, \ldots, \xi, \xi_n)=g(K_j,\xi^{p-2j-1}\cdot\xi_n),$$ 
for $0\le j\le \floor{\frac{p-1}{2}}$, where $\xi$ is as above a linear
combination with constant coefficients of the $g$-parallel vector fields $\xi_1,\ldots,\xi_{n-1}$. For
convenience we assume $\xi$ to have constant length one.
Note that $\alpha_0$ is constant since  $K_0$ is parallel. We want
to show that  all $\alpha_j$ have to vanish. First we find that
$$
(\d K_j)(\xi, \ldots, \xi) \; =\; - \,(\xi_n \lr \nabla_{\xi_n} K_j) (\xi, \ldots, \xi) \;=\; -\, \alpha'_j \ .
$$
Since $\xi$ is orthogonal to $\xi_n$ we have
$$
(\dd K_j)(\xi, \ldots, \xi) \; =\; (\xi_n \cdot \nabla_{\xi_n} K_j) (\xi, \ldots, \xi) \;=\; 0 \ .
$$
Taking the scalar product with $\xi^{p-2j+1}$ in \eqref{system} and using \eqref{power} we obtain for $0\le j\le \floor{\frac{p-1}{2}}$
\beq\label{eqj}
 \tfrac{(p-2j+1)(p-2j)}{n+2(p-2j-1)}  \, \alpha_j' \;+\; \tfrac{2j(p-2j+1)(p-2j)}{n+2(p-2j-1)} \, f'\,  \alpha_j
 \;=\;
 \tfrac{1}{n+2(p-2j+1)} \alpha_{j-1}'  \; +\;  \tfrac{n+2p-2j}{n+2(p-2j+1)} \, f'\,  \alpha_{j-1}\ .
\eeq
From this system of ODE's it will follow that the functions $\alpha_j$ are polynomials in $e^{-2f}$ of degree $j$.
Indeed we have

\begin{elem} \label{ode_system}
  Let $\set{b_j, c_j\suchthat 0 \le j \le l}$ be real constants and let $\set{f, \alpha_j \suchthat 0 \le j \le l}$ be a set of
  smooth real valued functions on an open interval $I \subset \R$ satisfying 
  following system of differential equations
  \begin{equation*}
    \alpha_j'  \;+\; 2 j \, f'\, \alpha_j   \;=\; b_j \,  \alpha_{j-1}' \;+\; c_j \,  f' \,  \alpha_{j-1}\ .
  \end{equation*}
Assume moreover that $\alpha_0$ is a constant. 
  Then every $\alpha_j$ is either identically zero or a polynomial in $e^{-2 f}$ of degree $j$.
\end{elem}
\begin{proof}
  The statement is proved by induction. It is true for $j=0$ since $\alpha_0$ is constant.
 We  set $\varphi = e^{2 f}$. Assume the statement to be true for all $k$ with $0 \le k \le j-1 < l$. Multiplying
  the equation for $\alpha_j$ with $\varphi^j$ and using $2 f' \varphi = \varphi'$ we get
  \begin{equation} \label{eq:multiplied_ode}
    \parens{\varphi^j \alpha_j}'
    \;=\; b_j  \varphi^j \alpha_{j-1}' \;+\; c_j f' \varphi^j \alpha_{j-1}.
  \end{equation}
  By assumption, either $\alpha_{j-1}=0$, or there is a polynomial $P$ of degree $(j-1)$ with  $\varphi^{j-1} \alpha_{j-1} = P ( \varphi)   $.
  In the first case, \eqref{eq:multiplied_ode} implies that $\varphi^j \alpha_j$ is constant and thus the induction hypothesis follows.
  Otherwise we have 
  \begin{equation*}
    \varphi^{j-1} \alpha_{j-1}'
      \;=\; 2 f' \varphi \,P' (\varphi) \;-\; 2 (j-1) f' \, P(\varphi)
      \;=\; f' \, Q (\varphi) \ ,
  \end{equation*}
  where $Q$ is a polynomial of degree $(j-1)$. Substituting this into
  (\ref{eq:multiplied_ode})  we arrive at
  \begin{equation*}
    \parens{\varphi^j \alpha_j}'
    \;=\; \tfrac{b_j}{2} \, Q ( \varphi) \, \varphi'
    \;+\; \tfrac{c_j}{2} \, P (\varphi ) \, \varphi'.
  \end{equation*}
  Finally, integrating this equation shows that $\varphi^j \alpha_j$ is polynomial
  of degree $j$ in $\varphi$,  proving the
  induction hypothesis for every $j \le l$.
\end{proof}

An easy calculation shows that $\Lambda(\xi^k\cdot\xi_n)=k(k-1)\xi^{k-2}\cdot\xi_n$ for $k\ge 2$ and therefore 
$$g(\LL^j \cdot K_j,\xi^{p-1}\cdot\xi_n)=d_jg(K_j,\xi^{p-2j-1}\cdot\xi_n)$$
for $0\le j\le \floor{\frac{p-1}{2}}$, where $d_j:=\frac{(p-1)!}{(p-2j-1)!}$. Moreover, $g(\LL^j \cdot K_j,\xi^{p-1}\cdot\xi_n)=0$ if $p$ is even and $j=\frac p2$. Taking the scalar product with $\xi^{p-1}\cdot\xi_n$ in the standard decomposition $K = \sum \LL^j K_j$ thus yields
$$K(\xi, \ldots, \xi, \xi_n)=\sum_{j=0}^{\floor{\frac{p-1}{2}}}d_j\alpha_j\ .$$
By Lemma \ref{properties} $(iv)$, the left hand side vanishes on the open set of points with $f' \neq 0$. 
Since by \eqref{eqj} and Lemma \ref{ode_system}, each $\alpha_j$ is either zero or a polynomial in $e^{-2f}$ of degree $j$, this equation can only hold if 
all $\alpha_j$ vanish identically. This shows that 
  \begin{equation}\label{kj}
K_j(\xi, \ldots, \xi, \xi_n)= 0,\qquad\forall j\le  \floor{\frac{p-1}{2}}
\end{equation}
on the open set where $f' \neq 0$.

Next we expand $K_0$ in powers of $\xi_n$, i.e. $K_0 = \sum_{j \ge 0} \xi_n^j \cdot P_j$, where every $P_j$ is a  polynomial
of degree $p-j$ in the vector fields $\xi_1,\ldots,\xi_{n-1}$. 
Moreover $P_1= 0$ by \eqref{kj}.
Applying the contraction $\L$ to the expansion of $K_0$ leads to the following relation 
of polynomials
  \begin{equation*}
    0 \;=\; \Lambda K_0 \;=\; \sum_{j \ge 0} \Lambda\, (\xi_n^j \cdot P_j)
    \;=\; \Lambda P_0 
      \;+\; \sum_{j \ge 2} j (j-1) \xi_n^{j-2} \cdot P_j
        \;+\; \xi_n^j \cdot \Lambda P_j \ .
  \end{equation*}
 Comparing the coefficients of $\xi_n^j$, this implies by immediate induction that $P_j = 0$ for every odd $j$ with $1 \le j \le p$. Hence
 $K_0$ is a polynomial in $\xi_1, \ldots, \xi_{n-1}$ and $\xi_n^2$. The same argument
 applies for all tensors $K_j$ from the standard decomposition $K = \sum_{j\ge 0} \LL^j \cdot K_j$. Hence we can write
 $
 K_j = \sum_{k\ge 0} P_{jk} \, \xi_n^{2k}
 $,
 where $P_{jk}$ are polynomials in $\xi_1, \ldots, \xi_{n-1}$ of degree $p-2j-2k$. By Lemma \ref{properties} $(i)$, the coefficients of $P_{jk}$ are constant in the variables $x_1,\ldots, x_{n-1}$. We now rewrite 
 \begin{equation}\label{q}
 K = \sum_{j,k \ge 0} \, \LL^j \cdot P_{jk} \cdot \xi_n^{2k} = \sum_{j,k\ge 0}\, \LL^j \cdot P_{jk}\cdot (\LL - \sum^{n-1}_{a=1}\xi_a^2 )^k
 = \sum_{k\ge 0} \tilde\LL^k \cdot Q_k  \ ,
 \end{equation}
 where $\tilde\LL = e^{-2f} \LL$ and $Q_k$ is a polynomial in $\xi_1, \ldots, \xi_{n-1}$ whose coefficients are constant in $x_1,\ldots,x_{n-1}$ but may depend on $x_n$. Since $ \nabla_{\xi_n} \tilde\LL=-2f'\tilde\LL$, we compute from \eqref{q}:
\begin{equation}\label{q1}
  \nabla_{\xi_n} K \;=\; \sum_{k\ge 0} \, (-2k \, f'\, \tilde\LL^k \cdot Q_k  \,+\, \tilde\LL^k \cdot \nabla_{\xi_n} Q_k)
 \;=\; \sum_{k\ge 0} \,\tilde\LL^k \cdot ( \nabla_{\xi_n} Q_k -2k \, f'\, Q_k ) \ .
 \end{equation}
 On the other side, since $\xi_n \lr \tilde\LL = 2 e^{-2f} \xi_n$, we have
\begin{equation}\label{q2}
 \dd f \lr K \;=\; f'\, \xi_n \lr K \;=\; \sum_{k\ge 0} \, 2k \, f' \, e^{-2f}\, \xi_n \cdot \tilde\LL^{k-1}\cdot Q_{k} \ .
 \end{equation}
Now, the original equation \eqref{k0} together with Lemma \ref{properties} yields $$\xi_n\cdot\nabla_{\xi_n}K=\dd K = - e^{2f}\tilde\LL \cdot \dd f \lr K$$ and taking \eqref{q}--\eqref{q2} into account, this leads to
 $$
\xi_n\cdot \sum_{k\ge 0} \tilde\LL^k \cdot ( \nabla_{\xi_n} Q_k -2k \, f'\, Q_k ) 
 \;=\;
 -\sum_{k\ge 0} \, 2k \, f' \,  \xi_n \cdot\, \tilde\LL^{k}\cdot  Q_{k} \ .
 $$
 It follows that $\xi_n\cdot \sum_{k\ge 0}  \tilde\LL^k \cdot \nabla_{\xi_n} Q_k = 0$ and thus 
 $\sum_{k\ge 0}  \tilde\LL^k \cdot \nabla_{\xi_n} Q_k = 0$. Since $Q_k$ are polynomials in $\xi_1, \ldots, \xi_{n-1}$, this leads to $\nabla_{\xi_n} Q_k = 0$
 for all $k$, i.e. $Q_k$ are polynomials in $\xi_1, \ldots, \xi_{n-1}$ with constant coefficients. It 
follows that $K$ itself is a polynomial in $\xi_1, \ldots, \xi_{n-1}$ and $\tilde\LL$ with constant coefficients on any connected component of the open set where $f'\neq 0$. But since the Killing equation is of finite type, two Killing tensors which coincide on some non-empty open set must coincide everywhere.
This proves the statement of Theorem \ref{main}. \hfill
$\Box$


\end{document}